\documentclass[reqno, english]{amsart}
\usepackage{a4wide}
\usepackage{enumerate}
\usepackage{amssymb,amsmath,amsfonts,amsthm}
\usepackage{hyperref}

\newtheorem{theorem}{Theorem}
\newtheorem{lemma}[theorem]{Lemma}
\newtheorem{corollary}[theorem]{Corollary}
\theoremstyle{definition}

\newtheorem{example}[theorem]{Example}
\newtheorem{remark}[theorem]{Remark}

\numberwithin{theorem}{section}
\numberwithin{equation}{section}
\numberwithin{table}{section}

\newcommand\ZZ{\mathbb{Z}}
\newcommand\NN{\mathbb{N}}
\newcommand\QQ{\mathbb{Q}}

\newcommand\OO{\mathcal{O}}
\newcommand\ppp{\mathfrak{p}}

\newcommand\PPP{\mathfrak{P}}

\newcommand\fff{\mathfrak{f}}

\newcommand\Kt{{\tilde{K}}}

\DeclareMathOperator\re{Re}
\DeclareMathOperator\N{N}
\DeclareMathOperator\Gal{Gal}
\DeclareMathOperator\Frob{Frob}

\newcommand\where{\: \mid\: }
\newcommand\qr[2]{\left(\frac{#1}{#2}\right)}
\newcommand\gp[1]{\langle #1\rangle}

\begin{document}

\title
[
Arithmetic progressions in binary quadratic forms and norm forms
]
{
Arithmetic progressions in\\ binary quadratic forms and norm forms
}

\date{26 October 2018}

\author{Christian Elsholtz}
\address{
Technische Universit\"at Graz\\
Institut f\"ur Analysis und Zahlentheorie\\
Kopernikusgasse 24/II\\
A-8010 Graz\\
Austria
}
\email{elsholtz@math.tugraz.at}

\author{Christopher Frei}
\address{
School of Mathematics\\
University of Manchester\\
Oxford Road, Manchester M13 9PL\\
United Kingdom
}
\email{christopher.frei@manchester.ac.uk}

\begin{abstract}
 We prove an upper bound for the length of an arithmetic progression represented by an irreducible integral binary quadratic form or a norm form, which depends only on the form and the progression's common difference. For quadratic forms, this improves significantly upon an earlier result of Dey and Thangadurai.
\end{abstract}


\maketitle

\section{Introduction}

\subsection{Binary quadratic forms}
De la Vall\'ee Poussin \cite{dlVPpt3, dlVPpt4} proved that the set of primes represented by an irreducible primitive integral binary quadratic form $F$ that is not negative definite has positive relative density in the primes\footnote{In fact, he considered only forms in which the coefficient of $xy$ is even. The general case was proved by Bernays in his dissertation \cite{zbMATH02625310}, where he also obtained an explicit error term.}. This fact, which is usually seen today as a well-known consequence of the Chebotarev density theorem, implies together with the Green-Tao theorem \cite{MR2415379} that each such quadratic form represents arithmetic progressions of arbitrary length. Our first main result provides an upper bound for the length of an arithmetic progression represented by $F$ in terms of $F$ and the progression's modulus.

Recall that the integral binary quadratic form $F=ax^2+bxy+cy^2$ is \emph{irreducible}, if it is irreducible over $\QQ$, and \emph{primitive}, if $\gcd(a,b,c)=1$. Its discriminant is $d=b^2-4ac$. Assume now that $F$ is irreducible and primitive, and that it is positive definite in case $d<0$.  Let $\Delta$ be the discriminant of the quadratic number field over which $F$ splits into linear factors. The \emph{conductor} of $F$ is the positive integer $f$ for which $d=f^2\Delta$.

\begin{theorem}\label{thm:qf_primitive}
  Let $F$ be an irreducible primitive integral binary quadratic form and let $\Delta$ be the discriminant of its splitting field. Assume that $F$ is positive definite if $\Delta<0$. Let
  \begin{equation*}
A=\{\ell+rg\where 0\leq r<k\},\ \ell\in\ZZ,\ g,k\in\NN,
\end{equation*}
be an arithmetic progression represented by $F$, i.e. ${\rm A}\subset F(\ZZ^2)$. Then
  \begin{equation*}
    k\leq C\left(\sqrt{|\Delta|}\log g  + |\Delta|^L\right),
\end{equation*}
  with absolute constants $C$ and $L$. Moreover, the value $L=7.999$ is admissible.
\end{theorem}

The conditions that $F$ be primitive and positive definite if $\Delta<0$ are no substantial restrictions. 
If $F$ is negative definite, replace it by $-F$, which represents arithmetic progressions of the same modulus and length. If $F$ is not primitive, replace it by $F/\gcd(a,b,c)$. This does not change the length of a progression, and it only decreases the modulus $g$. Some slight improvements to the bound in Theorem \ref{thm:qf_primitive} are recorded in Remark \ref{rem:improvement}.

Theorem \ref{thm:qf_primitive} improves in all aspects upon an earlier result of Dey and Thangadurai \cite{MR3295668}, which yields
\begin{equation}\label{eq:DT_bound}
  k < C \ell (g^2|d|)^{L_1}, 
\end{equation}
with a constant $L_1$, for which the value $15.6$ is admissible. The improvement is particularly visible when we regard $F$ as fixed, in which case, even if the dependence on $\ell$ in \eqref{eq:DT_bound} is ignored, our bound is essentially the logarithm of \eqref{eq:DT_bound}, or when $\Delta$ and $g$ are fixed, in which case our bound is constant whereas \eqref{eq:DT_bound} grows polynomially in $f$. Our proof of Theorem \ref{thm:qf_primitive} is not longer or less elementary than \cite{MR3295668}. 

\subsection{Norm forms}
Every irreducible integral binary quadratic form is a constant multiple of a norm form of a quadratic field (see \cite[\S 2.7]{MR0195803}). Our second main result is a far-reaching generalisation of  Theorem \ref{thm:qf_primitive} to arbitrary norm forms
\begin{equation}\label{eq:norm_form_no_scaling}
  \N(x_1,\ldots,x_n):=\N(x_1\omega_1+\cdots+x_n\omega_n),
\end{equation}
where $\omega_1,\ldots,\omega_n$ is a basis of a number field $K$. Since the length of an arithmetic progression is invariant under multiplication by constants, we rescale our  form $\N$ in an analogous way to how quadratic norm forms are rescaled to yield primitive integral binary quadratic forms. Denote the $\ZZ$-module generated by the basis $\omega_1,\ldots,\omega_n$ by
  \begin{equation*}
    M = \gp{\omega_1,\ldots,\omega_n},
  \end{equation*}
Then the $\OO_K$-module $\OO_KM$ generated by $\omega_1,\ldots,\omega_n$ is a fractional ideal of $K$ that contains $M$. The norm $\N(\OO_KM)$ is defined as the absolute value of the determinant of the matrix of basis change from a basis of $\OO_K$ to a basis of $\OO_KM$. It is independent of the choice of these bases. To the basis $\omega_1,\ldots,\omega_n$, we attach the form
  \begin{equation}\label{eq:def_norm_form} F(x_1,\ldots,x_n):=\frac{\N(x_1\omega_1+\cdots+x_n\omega_n)}{\N(\OO_KM)}\in\QQ[x_1,\ldots,x_n].
  \end{equation}
  If the basis $\omega_1,\ldots,\omega_n$ is replaced by the basis $a\omega_1,\ldots,a\omega_n$ for some $a\in \QQ$, $a>0$, this does not change the form $F$, as then both numerator and denominator are multiplied by $a^n$. We may thus assume without loss of generality that $M\subset\OO_K$, so $\OO_KM$ is an ideal of $\OO_K$ and $\N(\OO_KM)=[\OO_K:\OO_KM]$ is its absolute norm. This makes it clear that $F(\ZZ^n)\subset\ZZ$, since for all $\alpha\in M$ we have $\alpha\OO_K\subset\OO_KM$ and thus $\N(\OO_KM)\mid \N(\alpha\OO_K)=|\N(\alpha)|$. The following result is our generalisation of Theorem \ref{thm:qf_primitive} to arbitrary norm forms.

\begin{theorem}\label{thm:nf_general}
  Let $K$ be a number field of degree $n\geq 2$ with basis $\omega_1,\ldots,\omega_n$ and consider a form $F$ as in \eqref{eq:def_norm_form}. Let $\Kt$ be a normal closure of $K$ and denote its discriminant by $\Delta_\Kt$. Let
  \begin{equation*}
  A=\{\ell+rg\where 0\leq r<k\},\ \ell\in\ZZ,\ g,k\in\NN,
\end{equation*}
be an arithmetic progression represented by $F$, i.e. $A\subset F(\ZZ^n)$. Then
  \begin{equation}\label{eq:norm_ap_bound}
    k\leq C_{n}\left(|\Delta_\Kt|^5\log g  + |\Delta_\Kt|^L\right),
\end{equation}
with a constant $C_{n}$ depending at most on $n$ and an absolute constant $L$, for which one may take $L=694$.
\end{theorem}

Note that, up to the quality of the exponents on $|\Delta|$, Theorem \ref{thm:qf_primitive} is a special case of Theorem \ref{thm:nf_general}. Theorem \ref{thm:nf_general} also implies the same bound \eqref{eq:norm_ap_bound} for the length of an arithmetic progression represented by a usual norm form as in \eqref{eq:norm_form_no_scaling} coming from a basis $\omega_1,\ldots,\omega_n$ with $\omega_i\in\OO_K$. This is clear, since then $\N(\OO_K M)\geq 1$, so the re-scaling in \eqref{eq:def_norm_form} only decreases the modulus $g$ of an arithmetic progression, while keeping the length intact.

Our next result shows that, as in the case of binary quadratic forms, Theorem \ref{thm:nf_general} is non-trivial, in the sense that there is no upper bound depending only on $F$ for  the length of an arithmetic progression represented by $F$.

\begin{theorem}\label{thm:norm_arith}
  The form $F$, as in (\ref{eq:def_norm_form}), 
represents arithmetic progressions of arbitrary length.
\end{theorem}

It is not hard to see that Theorem \ref{thm:norm_arith} is implied by its special case where $\omega_1,\ldots,\omega_n$ is an integral basis of $K$. Moreover, in this case one can deduce Theorem \ref{thm:norm_arith} from \cite[Theorem 5.2]{MR3742196}, which, with the linear forms $f_i(u,v)=u+(i-1)v$ for $1\leq i\leq v$, provides an asymptotic formula for the number of suitably constrained norms representing $k$-term progressions.

In this note we follow a different approach to proving Theorem \ref{thm:norm_arith}. It will follow from the Green-Tao theorem \cite{MR2415379} combined with our next result, which states that a large class of norm forms as in \eqref{eq:def_norm_form} represent a positive proportion of the primes. This is certainly well known if the basis $\omega_1,\ldots,\omega_n$ is an integral basis of $K$, so $M=\OO_K$, and also when $[K:\QQ]=2$, in which case it states that an irreducible primitive integral binary quadratic form that is not negative definite represents a positive density of primes. We are not aware of a published result in our situation, which generalises both of these special cases simultaneously. Recall that the ring of multipliers $\OO=\{\alpha\in K\where \alpha M\subset M\}$ of $M$ is an order of the number field $K$.

\begin{theorem}\label{thm:norm_primes}
  Suppose that $M$ is an invertible ideal of its ring of multipliers $\OO$. Then the set of primes represented by the form $F$ defined in \eqref{eq:def_norm_form} has positive relative density.
\end{theorem}

If $[K:\QQ]=2$, then every ideal of $\OO$ with ring of multipliers $\OO$ is invertible, see \cite[Corollary 4.4]{conrad}. However, this can fail in degree $3$. In some situations, we can extend the conclusion of Theorem \ref{thm:norm_primes} to ideals that are not invertible. Our choice of the number field, order
and ideal in the following example was motivated by  \cite[Example 3.5]{conrad}.

\begin{example}\label{ex:fail}
  Let $K=\QQ(\sqrt[3]{2})$ and choose the basis
  $8,2\sqrt[3]{2},2\sqrt[3]{2}^2$. Then
  $M=\gp{8,2\sqrt[3]{2},2\sqrt[3]{2}^2}$. The ring of
  multipliers of $M$ is $\OO=\ZZ+2\OO_K=\gp{1,2\sqrt[3]{2},2\sqrt[3]{2}^2}$,
  and $M$ is not an invertible ideal of $\OO$. We have $\OO_KM=\gp{4,2\sqrt[3]{2},2\sqrt[3]{2}^2}$ with $\N(\OO_KM)=16$, so the corresponding norm
  form is
  \begin{eqnarray*}
    F(x,y,z)& =&\frac{\N(8x+2\sqrt[3]{2}y+2\sqrt[3]{2}^2z)}{16}=32x^3+y^3+2z^3-12xyz.
  \end{eqnarray*}
Even though $M$ is not invertible, we can use Theorem \ref{thm:norm_primes} to show that $F$ represents a positive density of the primes.
\end{example}

\begin{corollary}\label{cor:explicit_form}
The integral form $32x^3+y^3+2z^3-12xyz$ represents a positive density of primes. 
\end{corollary}
It would be interesting to see how far this can be generalised, and whether one can eliminate the restriction that $M$ be invertible from Theorem \ref{thm:norm_primes}.

\section{A congruence argument modulo $p^2$}

Let $S\subset\NN$ be a subset and let $\mathcal{P}_S$ be the set of primes $p$ that satisfy
\begin{equation}\label{eq:def_PS}
  \text{for all $n\in S$: if } p\mid n \text{ then }p^2\mid n.
\end{equation}

\begin{lemma}\label{lem:key}
  Let $\ell\in\ZZ$, $k,g\in\NN$, and suppose that $\ell+rg\in S$ holds for all 
$r\in\{0,\ldots,k-1\}$. Then any $p\in\mathcal{P}_S$ with 
$2p\leq k$ satisfies $p\mid g$.
\end{lemma}

\begin{proof}
  If $p\nmid g$, then the progression $\ell, \ell + g, \ell + 2g,\ldots, \ell + (2p-1)g$ will cover every residue class modulo $p$ exactly twice. Hence there exists $0 \leq i \leq p-1$ such that $\ell + ig \equiv \ell + (i + p)g \equiv 0 \bmod p$, but that at least one of $\ell + ig$ or $\ell + (i + p)g$ is not congruent to $0 \bmod p^2$, a contradiction to $p \in \mathcal{P}_S$.
\end{proof}

\begin{lemma}\label{lem:estimate}
  Assume that for $C_1,C_2>0$ we have an inequality
  \begin{equation}\label{eq:density}
    \sum_{\substack{p\in\mathcal{P}_S\\p\leq X}}\log p\geq \frac{X}{C_1},\quad\text{ for all }X\geq C_2.
  \end{equation}
  Then any arithmetic progression $\{\ell+rg\where 0\leq r<k\}\subset S$ 
satisfies
  \begin{equation*}
    k\leq 2\max\left\{C_1\log g,\: C_2\right\}.
  \end{equation*}
\end{lemma}

\begin{proof}
  We may suppose $k\geq 2C_2$, as otherwise there is nothing to prove. We define
  \begin{equation*}
  P:=\prod_{\substack{p\in\mathcal{P}_S\\p\leq k/2}}p.
\end{equation*}
Then $P\mid g$ by Lemma \ref{lem:key}. Hence, 
\begin{align*}
  g\geq P = \exp\Big(\sum_{\substack{p\in\mathcal{P}_S\\p\leq k/2}}\log p\Big)\geq 
\exp\left(\frac{k}{2C_1}\right),
\end{align*}
and thus $k\leq 2C_1\log g$.
\end{proof}

\section{Binary quadratic forms}
Let $F=ax^2+bxy+cy^2$ be a primitive irreducible integral binary quadratic form that is not negative definite, and write $d=f^2\Delta$, where $d=b^2-4ac$ is the discriminant of $F$,
$f$ is its conductor, and $\Delta$ is the discriminant of its splitting field. We are interested in the value set $S=F(\ZZ^2)$ and the corresponding set $\mathcal{P}_S$ of primes $p$ defined by \eqref{eq:def_PS},
i.e.\ $\text{for all $n\in S$: if } p\mid n \text{ then }p^2\mid n.$

\begin{lemma}\label{lem:qr_cond}
  There is a subset $T\subset\left(\ZZ/\Delta\ZZ\right)^\times$ of cardinality $\varphi(|\Delta|)/2$, 
  such that $\mathcal{P}_S$ contains every odd prime $p$ that satisfies $(p\bmod\Delta)\in T$.
\end{lemma}

\begin{proof}
For every odd prime $p\mid d$, the quadratic form $F$ is degenerate modulo $p$ and thus of the form $F(x,y)\equiv A L(x,y)^2\bmod p$, for some $A\in\ZZ$ and a linear form $L(x,y)$ with integral coefficients. Since $p$ can not divide all coefficients of $F$, we get $A\not\equiv 0\bmod p$. Hence, for $x,y\in\ZZ$,
  \begin{equation*}
    p\mid F(x,y)\Longleftrightarrow p\mid L(x,y)\Longleftrightarrow p^2\mid F(x,y), 
  \end{equation*}
  and thus $p\in \mathcal{P}_S$. Assume now that $p\nmid d$. Then, with $\qr{\cdot}{\cdot}$ denoting the Kronecker symbol,
\begin{equation*}
  \qr{d}{p}=\qr{f^2\Delta}{p}=\qr{\Delta}{p}.
\end{equation*}
Since the Kronecker symbol $\qr{\Delta}{\cdot}$ is a non-principal quadratic Dirichlet character modulo $\Delta$ for every fundamental discriminant $\Delta$, the complement $T\subset (\ZZ/\Delta\ZZ)^*$ of its kernel has cardinality $\varphi(|\Delta|)/2$. If $(p\bmod \Delta)\in T$, then $d$ is not a square modulo $p$, so $F$ is not isotropic modulo $p$, which means exactly that $F(x,y)\equiv 0\bmod p$ implies $x,y\equiv 0\bmod p$, and thus $F(x,y)\equiv 0\bmod p^2$.
\end{proof}

\begin{lemma}\label{lem:qf_density_inexplicit}
  The set $S=F(\ZZ^2)$ satisfies \eqref{eq:density} with $C_1\ll\sqrt{|\Delta|}$ and $C_2\ll|\Delta|^L$, for some $L>0$.
\end{lemma}

\begin{proof}
  Using Lemma \ref{lem:qr_cond}, it is enough to show that
  \begin{equation}\label{eq:linnik2}
    \sum_{\substack{p\leq X\\p\equiv c\bmod \Delta}}\log p\gg \frac{X}{\varphi(|\Delta|)\sqrt{|\Delta|}}
  \end{equation}
  holds for all $c$ with $\gcd(c,\Delta)=1$, whenever $X\gg|\Delta|^L$. For $|\Delta|\ll 1$, this follows from the prime number theorem in arithmetic progressions. For all large enough $|\Delta|$, on the other hand, this is a quantitative version of Linnik's theorem on the least prime in an arithmetic progression, see \cite[Corollary 18.8]{IK04}.
\end{proof}

\subsection*{Proof of Theorem \ref{thm:qf_primitive}}
To prove the upper bound on $k$ apply Lemma \ref{lem:estimate} with $C_1, C_2$ from Lemma \ref{lem:qf_density_inexplicit}.
We now show that $L=7.999$ is admissible.
For large enough $|\Delta|$ and $x\geq |\Delta|^{7.999}$, \cite[Proposition 5]{Maynard2013} implies that
\begin{equation*}
    \sum_{\substack{p\leq X\\p\equiv c\bmod \Delta}}\log p\gg \frac{\lambda X}{\varphi(|\Delta|)}.
\end{equation*}
Here, $\lambda\gg 1$ if no Dirichlet $L$-function $L(s,\chi)$, for any Dirichlet character $\chi$ modulo $\Delta$, has a real zero in the range $1-\eta/(\log |\Delta|)\leq \re z\leq 1$, for some sufficiently small absolute constant $\eta>0$. Otherwise, if such a zero $z$ exists, then it is necessarily unique, and we take $\lambda\in (0,\eta]$ such that $z=1-\lambda/(\log |\Delta|)$. In any case, we have $\lambda\gg 1/\sqrt{|\Delta|}$ (see \cite[Theorem 3]{zbMATH03480705}), and thus \eqref{eq:linnik2}. Use this instead of \cite[Corollary 18.8]{IK04} in the proof of Lemma \ref{lem:qf_density_inexplicit}.

\begin{remark}\label{rem:improvement}
Pintz's result \cite[Theorem 3]{zbMATH03480705} provides in fact the bound $\lambda\gg \log(|\Delta|)/\sqrt{|\Delta|}$. Moreover, one can use Siegel's result $\lambda\gg_{\epsilon} |\Delta|^{-\epsilon}$, which holds for $\epsilon>0$  with an ineffective implied constant (see \cite[Theorem 5.28]{IK04}). These lead to slight improvements in Theorem \ref{thm:qf_primitive}, giving the bounds       \begin{equation*}
    k\ll \frac{\sqrt{|\Delta|}}{\log |\Delta|}\log g + |\Delta|^L\quad\text{ and }\quad k\ll_{\epsilon}|\Delta|^{\epsilon}\log g + |\Delta|^L,
  \end{equation*}
  respectively, where the implied constant in the second bound is ineffective for $\epsilon<1/2$.
\end{remark}

\section{Arbitrary norm forms}

In this section we prove Theorem \ref{thm:nf_general}. Let  $K$ be a number field of degree $n\geq 2$ and $F$ a form as in \eqref{eq:def_norm_form}. Consider the value set
\begin{equation*}
  S=F(\ZZ^n) = \{\N(\alpha)/\N(\OO_KM)\where \alpha\in M\}
\end{equation*}
and the corresponding set of primes $\mathcal{P}_S$ defined in \eqref{eq:def_PS}. As explained before the statement of Theorem \ref{thm:nf_general}, we may assume that $M\subset\OO_K$, so $\OO_K M$ is an ideal of $\OO_K$.

\begin{lemma}\label{lem:nf_cond}
Let $p$ be a prime such that all prime ideals $\ppp$ of $\OO_K$ above $p$ satisfy $p^2\mid\N(\ppp)$. Then $p\in\mathcal{P}_S$.
\end{lemma}

\begin{proof}
  Let $p$ be a prime satisfying the hypothesis of the lemma and $m\in S$ with $p\mid m$. Then $m=\N(\alpha)/\N(\OO_KM)=\N(\alpha(\OO_KM)^{-1})$ for some $\alpha\in M\subset\OO_KM$. Since $p\mid m$, there is a prime ideal $\ppp$ of $\OO_K$ above $p$ that divides the ideal $\alpha(\OO_KM)^{-1}$, and thus $p^2\mid\N(\ppp)\mid\N(\alpha(\OO_KM)^{-1})=m$.
\end{proof}

Recall that $\Kt$ is a normal closure of $K$. For every prime $p$ that is unramified in $\Kt$, we write $\Frob_p\subset\Gal(\Kt/\QQ)$ for the conjugacy class of Frobenius automorphisms of prime ideals above $p$ in $\OO_\Kt$. 

\begin{lemma}\label{lem:nf_density}
  There is a non-empty union $U$ of conjugacy classes of $\Gal(\Kt/\QQ)$, such that $\mathcal{P}_S$ contains every prime $p$ that is unramified in $\Kt$ and satisfies $\Frob_p\subset U$.
\end{lemma}

\begin{proof}
Write $G=\Gal(\Kt/\QQ)$ and $H=\Gal(\Kt/K)\subsetneqq G$. Let $U$ be the union of all conjugacy classes $C$ of $G$ with $C\cap H=\emptyset$. By \cite[Lemma 13.5]{neukirchANT}, all unramified primes $p$ with $\Frob_p\subset U$ satisfy the hypothesis of Lemma \ref{lem:nf_cond} and are thus in $\mathcal{P}_S$. Since $H$ is a proper subgroup of $G$, its conjugate subgroups do not cover $G$. Indeed, the number of conjugate subgroups of $H$ is at most $[G:H]$ by the orbit-stabiliser theorem applied to the action of $G$ on its subgroups by conjugation. On the other hand, the union of these conjugate subgroups is not disjoint. Hence, there is $\sigma\in G$ all of whose conjugates are in $G\smallsetminus H$, which shows that $U\neq\emptyset$.
\end{proof}

\subsection*{Proof of Theorem \ref{thm:nf_general}}
A result of Thorner and Zaman \cite[(3.2)]{zbMATH06748168} guarantees that
\[ \sum_{\substack{p\leq X\\\Frob_p\subset U}}1 \geq \delta \frac{X}{\log X},\]
where $\delta\gg_n\frac{1}{|\Delta_\Kt|^5}$, under the conditions that $|\Delta_\Kt|$ is large enough and $X\gg_n |\Delta_\Kt|^L$, for a positive constant $L$ (where $L=694$ is admissible).   For $|\Delta_\Kt|\ll_n 1$, the same bound follows from the Chebotarev density theorem. Using this and Lemma \ref{lem:nf_density}, we see that 
  \begin{equation*}
    \sum_{\substack{p\leq X\\p\in\mathcal{P}_S}}\log p
\geq\sum_{\substack{p\leq X\\\Frob_p\subset U}}\log p
\geq \sum_{\substack{\text{ the first } \delta X/\log X\\\text{primes } p}}\log p
\gg_{n} \frac{X}{|\Delta_\Kt|^5}.
\end{equation*}
Here the estimate of the prime number theorem that the $n$-th prime
has size $p_n \sim n \log n$
was used. Hence, we can apply Lemma 
\ref{lem:estimate} with 
$C_1\ll_{n}|\Delta_\Kt|^5$, $C_2\ll_n|\Delta_\Kt|^L$.

\section{Class field theory}
In this section, we prove Theorem \ref{thm:norm_arith} and Theorem \ref{thm:norm_primes}, as well as Corollary \ref{cor:explicit_form}. Let $K$ be a number field and $\OO$ an order of $K$ of conductor $\fff$. We
construct a class field similar to the ring class field of $\OO$ studied in
\cite{MR3368170}, except that our congruence subgroup also incorporates the
condition that $\N(\alpha)>0$. We consider the modulus $\fff\infty$
of $\OO_K$, where $\infty$ is the product of all real places of $K$. Let $I_{K}^{\fff}$ be the subgroup of
fractional $\OO_K$-ideals of $K$ generated by the prime ideals not dividing
$\fff$, and let $P_{K,\OO}^{\fff,+}$ be the subgroup generated by all principal ideals
\begin{equation*}
  \alpha\OO_K\quad\text{ with }\quad\alpha\in\OO,\ \alpha\OO+\fff=\OO,\ N(\alpha)>0.
\end{equation*}

We moreover let $I(\OO,\fff)$ denote the subgroup of invertible ideals of $\OO$
that is generated by all prime ideals $P$ with $P+\fff=\OO$, and $P(\OO,\fff)^+$
the subgroup generated by all principal ideals 
\begin{equation*}
  \alpha\OO\quad \text{ with }\alpha\in\OO,\ \alpha\OO+\fff=\OO,\ \N(\alpha)>0.
\end{equation*}
\begin{lemma}\label{lem:subgroup_properties}\ 
  \begin{enumerate}
  \item The group $P_{K,\OO}^{\fff,+}$ is a congruence subgroup modulo
    $\fff\infty$.
  \item The map $I(\OO,\fff)\to I_K^{\fff}$ defined by $I\mapsto I\OO_K$ is an
    isomorphism.
  \item The isomorphism from (2) induces an isomorphism
    $I(\OO,\fff)/P(\OO,\fff)^{+}\to I_K^\fff/P_{K,\OO}^{\fff,+}$.  
  \end{enumerate}
\end{lemma}

\begin{proof}
  (1): it is obvious that $P_{K,\OO}^{\fff,+}\subset I_{K}^{\fff}$. We need to show that
  $P_{K,\OO}^{\fff,+}$ contains all principal ideals $\alpha\OO_K$ with
  $\alpha\in\OO_K$, $\alpha\equiv 1\bmod\fff$ and $\sigma(\alpha)>0$ for all
  real embeddings $\sigma$ of $K$. Any such $\alpha$ clearly satisfies
  $\N(\alpha)>0$ and, moreover, $\alpha\in 1+\fff$ implies $\alpha\in\OO$ and
  $\alpha\OO+\fff=\OO$. Hence, $\alpha\in P_{K,\OO}^{\fff,+}$.
  
  (2): this is \cite[Proposition 3.4]{MR3368170}.

  (3): the isomorphism from (2) maps $P(\OO,\fff)^+$ to $P_{K,\OO}^{\fff,+}$.
\end{proof}

By the existence theorem of class field theory, there is a unique Abelian
extension $K_{\OO}^{+}/K$, all of whose ramified primes divide $\fff\infty$, for which the Artin map
\begin{equation*}
  \psi_{\fff} : I_{K}^{\fff}\to \Gal(K_\OO^{+}/K)
\end{equation*}
is surjective and has kernel $P_{K,\OO}^{\fff,+}$.
Let $L$ be the normal closure of $K_{\OO}^{+}$. For any prime $p$ unramified in $L/\QQ$, we denote
its Frobenius class in $\Gal(L/\QQ)$ by $\Frob_{L/\QQ,p}$.

\begin{lemma}\label{lem:CFT}
  Let $I\subset\OO$ be an
  ideal with $I+\fff=\OO$. Suppose that the prime $p$ is unramified in $L$, satisfies $p\OO_K + \fff=\OO_K$, and moreover there is $\sigma\in\Frob_{L/\QQ,p}$ with
  \begin{equation*}
\sigma|_{K_{\OO}^{+}}=\psi_\fff(\OO_KI^{-1})\in \Gal(K_{\OO}^{+}/K).
\end{equation*}
  Then there is $\alpha\in I$ with $\N(\alpha)=\N(\OO_KI)p$.
\end{lemma}

\begin{proof}
Let $\PPP$ be a prime ideal of $\OO_L$ with
  $\Frob_{\PPP/p}=\sigma$. Then $\sigma(x)\equiv x^p\bmod\PPP$ for all
  $x\in\OO_L$. Let $\ppp=\PPP\cap K$. Since $\sigma\in\Gal(L/K)$, we get
  $x=\sigma(x)\equiv x^p\bmod\ppp$ for all $x\in\OO_K$, and thus $\ppp$ is a
  degree-$1$-prime ideal of $\OO_K$ above $p$. Hence, $\N(\ppp)=p$, which shows
  moreover that $\sigma(x)\equiv x^{\N(\ppp)}\bmod\PPP$ for all $x\in\OO_L$, so
  $\sigma=\Frob_{\PPP/\ppp}\in\Gal(L/K)$. Hence,
  $\psi_\fff(\ppp)=\Frob_{K_{\OO}^{+}/K,\ppp}=\sigma|_{K_{\OO}^{+}}=\psi_\fff(\OO_KI^{-1})$,
which shows that the classes of $\ppp$ and $\OO_KI^{-1}$ in
$I_K^\fff/P_{K,\OO}^{\fff,+}$ coincide. By Lemma \ref{lem:subgroup_properties},
the classes of $P=\ppp\cap\OO$ and $I^{-1}$ in $I(\OO,\fff)/P(\OO,\fff)^{+}$
coincide as well, and thus $\alpha I^{-1}=P$ for some $\alpha\in
P(\OO,\fff)^{+}$. In particular, $\alpha\in\alpha\OO=\alpha I^{-1}I=PI\subset I$. Since $\N(\alpha)>0$, we get $\N(\alpha)=\N(\alpha\OO_K)=\N(\alpha
\OO_KI^{-1}\OO_K I)=\N(\alpha \OO_K I^{-1})\N(\OO_K I)=\N(\ppp)\N(\OO_K I)=p\N(\OO_K I)$, as desired.
\end{proof}

\subsection*{Proof of Theorem \ref{thm:norm_primes}}
Assume that $M$ is an invertible ideal of its ring of multipliers $\OO$, and let $\fff$ be the conductor ideal of the order $\OO$. By \cite[Theorem 3.11]{MR3368170}, the invertible ideal $M$ is of the form
$M=\gamma I$, where $\gamma\in K\smallsetminus\{0\}$ and $I$ is an ideal of
$\OO$ with $I+\fff=\OO$. By Lemma \ref{lem:CFT} and the Chebotarev density
theorem, a positive density of the primes have the form $p=\N(\alpha)/\N(\OO_KI)$ with
$\alpha\in I$. Multiplying $\alpha$ and $I$ by $\gamma$ proves the result.

\subsection*{Proof of Corollary \ref{cor:explicit_form}}
Let $M$ and $\OO$ be as in Example \ref{ex:fail}.
The element $\gamma = 2\sqrt[3]{2}\in M$ satisfies $\N(\gamma)=16$. Since
$\gamma\OO$ is invertible, Theorem \ref{thm:norm_primes} yields a positive
density of primes of the form $\N(\alpha)/16$, for
$\alpha\in\gamma\OO\subset M$. A fortiori, there is a positive density of
primes of the form $\N(\alpha)/16$ for $\alpha\in M$.

\subsection*{Proof of Theorem \ref{thm:norm_arith}}
Let $\OO$ be the ring of multipliers of $M$. Since $\OO$ is also the ring of multipliers of all modules $\alpha M$, for $\alpha\in K$, we may multiply all our basis elements $\omega_i$ by an appropriate integer and thus assume from now on that $M\subset\OO$ is an ideal of $\OO$. Let $\gamma\in M$,
then $\gamma\OO$ is an invertible ideal of $\OO$ (with ring of multipliers
$\OO$). By Theorem \ref{thm:norm_primes}, a positive proportion of the primes
have the form $\N(\alpha)/\N(\gamma\OO_K)$ for $\alpha\in\gamma\OO\subset M$. By
the Green-Tao theorem \cite{MR2415379}, the set of values $\N(\alpha)/\N(\gamma\OO_K)$, for
$\alpha\in M$, contains arithmetic progressions of arbitrary length. Hence, the
same holds for the set of values $\N(\alpha)/\N(\OO_KM) =
[\OO_KM:\gamma\OO_K]\cdot\N(\alpha)/\N(\gamma\OO_K)$.

\subsection*{Acknowledgement}
We thank the anonymous referee for pointing out a minor inaccuracy in a previous version of the proof of Theorem \ref{thm:qf_primitive}.

\bibliographystyle{plain}
\bibliography{bibliography}
\end{document}